%% file: main.tex
\documentclass[final, 12pt]{article}
\usepackage{graphicx} 
\usepackage{tikz}
\usetikzlibrary{trees, calc, decorations.pathreplacing}
\usetikzlibrary {arrows.meta}
\usepackage{amsmath, amssymb, amsthm}
\usepackage{mathdots}
\usepackage{standalone}
\usepackage{changepage}
\usepackage{authblk}
\usepackage[most]{tcolorbox}
\usepackage{enumitem}
\usepackage{cleveref}
\usepackage{optidef}
\usepackage{xspace}
\usepackage{algorithm}
\usepackage{algorithmic}
\usepackage{csquotes}
\usepackage[english]{babel}
\usepackage{amsfonts}
\usepackage{multirow}
\usepackage{bigstrut}
\usepackage{etoolbox}
\usepackage[nolist]{acronym}

\usepackage{nomencl}

\makenomenclature
\usepackage[backend=biber, sorting=none]{biblatex}

\title{Complexity Analysis of a Bicriteria Directed Multimodal Transportation Network Design Problem}
\author{Dominik Leib\thanks{dominik.leib@itwm.fraunhofer.de} }
\author{Susanne Fritzler}
\author{Neele Leithäuser}
\affil{Fraunhofer Institute of Industrial Mathematics ITWM}

\usepackage{biblatex}
\newtheorem{theorem}{Theorem}[section] 
\newtheorem{lemma}[theorem]{Lemma} 
\newtheorem{corollary}[theorem]{Corollary}
\theoremstyle{definition}
\newtheorem{definition}[theorem]{Definition}

\newenvironment{decisionproblem}[3]
  {\vspace{8pt}
  \begin{adjustwidth}{0.5cm}{}
  \textsc{#1} \\[1ex]
   \noindent\textit{Instance:} #2 \\[1ex]
  \noindent\textit{Question:}  #3 
  \end{adjustwidth}
    \vspace{8pt}}

\newenvironment{searchproblem}[3]
  {\vspace{8pt}
  \begin{adjustwidth}{0.5cm}{}
  \textsc{#1} \\[1ex]
   \noindent\textit{Instance:} #2 \\[1ex]
  \noindent\textit{Task:}  #3 
  \end{adjustwidth}
    \vspace{8pt}}

\DeclareMathOperator{\Eta}{\mathrm{H}}
\DeclareMathOperator{\Tau}{\mathrm{T}}

\DeclareMathOperator*{\argmin}{arg\,min}
\DeclareMathOperator{\budgetB}{\beta}
\DeclareMathOperator{\budgetR}{\gamma}
\DeclareMathOperator{\dist}{\mathrm{dist}}
\DeclareMathOperator{\vv}{\mathrm{\textbf{v}}}

\newcommand{\sol}{X}
\newcommand{\terms}{T}
\newcommand{\NP}{\textbf{NP}}
\newcommand{\ie}{i.\,e.\@\xspace}
\newcommand{\eg}{e.\,g.\@\xspace}

\newcommand{\wrt}{w.r.t.\@\xspace}

\begin{document}

\begin{acronym}
	\acro{STP}{Steiner Tree Problem}
	\acro{ESUM}{Exact Subset Sum Problem}
        \acro{SSUM}{Subset Sum Problem}
	\acro{MSP}{Modal Split Problem}
        \acro{MSPR}{Relaxed Modal Split Problem}
	\acro{NDP}{Network Design Problem}
	\acro{DiNDP}{Directed Network Design Problem}
        \acro{DiSTP}{Directed Steiner Tree Problem}
        \acro{X3C}{Exact-3-Cover}
        \acro{KPS}{Knapsack Problem}
        \acro{UKPS}{Unbounded Knapsack Problem}
        \acro{MKPS}{Multidimensional Knapsack}
\end{acronym}

\maketitle

\begin{abstract}
    \input{abstract}   
\end{abstract}

\section{Introduction}
    \input{introduction}

\section{Model} \label{sec:model}
    \input{model}

\section{Complexity Analysis} \label{sec:complexityanalysis}

\input{complexity}

\section{Approximability of Pareto Optimal Solutions} \label{sec:approximability}
    \input{frontier}

\section*{Declaration of competing interest}

The authors declare that they have no known competing financial interests or personal relationships that could have appeared to influence the work reported in this paper.

\section*{Acknowledgement}
    \input{acknowledgements}

\printbibliography

\appendix

\section{Appendix}

\input{appendix}

\input{glossary}
\printnomenclature

\end{document}

%% file: abstract.tex
In this paper, we address a bicriteria network design problem that arises from practical applications in urban and rural public transportation planning. We establish the problem's complexity and demonstrate inapproximability results, highlighting the inherent difficulties in finding optimal solutions. Additionally, we identify special cases where approximability can be achieved, providing valuable insights for practitioners. Our proofs leverage complexity results related to directed network design problems, an area that has received limited attention in the existing literature. By investigating these complexity results, we aim to fill a critical gap and enhance the understanding of the interplay between bicriteria decision-making and network design challenges. 

%% file: introduction.tex
Network design plays a crucial role in optimizing transportation systems, where the efficient allocation of resources significantly impacts overall performance. This field leverages mathematical techniques such as graph theory, linear programming, and combinatorial optimization to address complex challenges, including traffic flow, connectivity, and cost minimization. These methodologies provide a foundational framework for tackling the dynamic challenges inherent in network design, offering insights into efficient system architecture and resource management. Early references to network design include \cite{Wong84} and \cite{wong78}, while the complexity of such problems was established in \cite{LenstraKan1979} and \cite{wong80}. Modern references, such as \cite{crainic21} and \cite{farahani2013review}, provide detailed overviews of specific variants.

In public transportation systems, user convenience \textendash{} such as minimizing travel time and reducing detours \textendash{} often conflicts with the goal of energy reduction. Efficient routes that prioritize direct travel and frequent service can lead to increased energy consumption due to higher operational demands and less optimized scheduling. Conversely, energy-efficient strategies might involve fewer trips, longer wait times, or less direct routes, inconveniencing passengers. This inherent conflict necessitates the use of a bicriteria optimization approach. Since no single design is in general optimal across all objectives, bicriteria optimization provides a framework to evaluate trade-offs and identify balanced solutions. Multicriteria transportation problems have a long history due to the multiobjective nature of the problem. See \eg \cite{FRIESZ199344} for a very general $4$-criteria model solved by simulated annealing, or \cite{JHA2019166} for a metaheuristic approach to a multicriteria transit network design problem. Additionally, \cite{CURRENT1986187} delivers a review on early models and approaches to multicriteria transportation network design. Lastly, designing and planning public transportation networks at both tactical and operational levels is a multifaceted task involving numerous stages of varying difficulty. Multiple modes of transportation may also need to be considered to account for effects across them. Examples for works about multimodal transportation planning are \cite{cantarella06}, where a bimodal, multicriteria problem has been solved by a genetic algorithm or \cite{HOOGERVORST2024106640}, where segments of an existing bus line are upgraded to a rapid variant. In \cite{farahani2013review}, also an overview about further works on this topic is contained.

For this we propose a model in \Cref{sec:model} to provide a coarse indication of where to install which mode of transportation. The model aims to achieve energy reduction goals in a Pareto-optimal sense, i.e., while simultaneously respecting user convenience, under a given budget for investment. The number of users traveling, as well as their start and destination points, is fixed. However, the model determines the routes to take and the mode of transportation used on each connection. The investment strategy derived from the model's solution can then guide the detailed planning stages of transport network design. In \Cref{sec:complexityanalysis}, we investigate the computational complexity of this problem and establish its \NP-completeness, even in various strongly simplified cases. For this, we build on similar results for classical network design. However, due to the lack of results for the directed case, we establish it as well. Additionally, we demonstrate the inapproximability of the problem in the general case, meaning that no efficient algorithm can guarantee solutions within a certain approximation ratio. As a side result, we also transfer an inapproximability result from undirected to directed cases, with proof provided in the appendix. Finally, in \Cref{sec:approximability}, we demonstrate that our problem remains \NP-complete even when passenger flow is fixed. We also present an approximation strategy for this case to sample points on the Pareto frontier.

%% file: model.tex
In the following, we propose a bicriteria network design model to study trade-offs between energy-efficient and convenient solutions. The model adopts the perspective of a provider engaged in tactical planning, who allocates a fixed budget to install public transportation modes on a predefined base network of possible connections. As a baseline, we assume that all passengers initially travel independently using individual transport. The provider's goal is to install multiple vehicles from a set of available modes and assign passengers to these modes.
The introduction of public transportation can reduce energy consumption due to consolidation effects\textendash\ie, the per-head energy consumption decreases as vehicle occupancy increases. However, achieving consolidation often requires passengers to accept detours compared to their direct routes via individual transport. Additionally, public transportation modes may be slower overall, further reducing user convenience.
To evaluate the trade-offs between these competing factors, we consider two primary objectives: Enhancement of user convenience by minimizing travel time and reduction of energy consumption, both under a given investment budget.
Both objectives are interrelated and critical for designing sustainable and efficient transportation networks. Promoting the use of public transport systems is a key strategy for achieving energy reduction, as it consolidates travelers into higher-capacity vehicles, thereby lowering per-head energy consumption and optimizing the network's overall energy efficiency.
To realize these benefits, it is essential to achieve a sufficient degree of occupancy within public transport vehicles. High occupancy not only enhances the economic viability of public transportation but also maximizes its environmental benefits by reducing energy use per traveler. Balancing these objectives\textendash user convenience and energy efficiency\textendash requires careful planning and resource allocation, which our proposed model aims to address.

Thus the \ac{MSP} is defined as follows: We are given a directed, weighted graph $G = (V,E,w)$, a travel \emph{demand} $D: V^2 \to \mathbb{N}_{\geq 0}$ with $D(v,v) = 0$ and a budget $B$ for public transport. The pairs $(s, t) \in V^2$ with $D(s, t) > 0$ are the \emph{commodities} $K$. We are given multiple options for public transport, where each public transport device is uniquely defined by its travel time $\tau_i$ per unit of distance, energy consumption $\eta_i$ and costs $c_i$ per unit of distance and vehicle and a capacity $k_i$ per vehicle installed for $i = 1,\ldots,m$. A \emph{layout} $L: E \to \mathbb{N}^m$ maps to each edge the number of assigned public transportation vehicles of mode $i$, it is feasible if $\sum_e w(e) \sum_{i=1}^m c_i L(e)_i \leq B$, \ie if the budget is respected. We don't assume the layout to be spanning, instead all passengers, which are not covered by public transport are assumed to travel independently by private transport, which will be mode $0$. Private transport has travel time and energy consumption per unit of distance $\tau_0, \eta_0$ as well, but costs of zero and is for simplicity assumed to be allowed on each edge without any capacity constraints. Individual transport is considered to be continuous, but we implicitly assume $k_0 = 1$. The movement of the passengers is described by a $s$-$t$-flow $F_{\kappa}$ for each commodity $\kappa = (s,t) \in K$, which is a function $F_{\kappa}: E \to \mathbb{R}_{\geq 0}$ s.t. $\sum_{e \in in(v)} F_{\kappa}(e)- \sum_{e \in out(v)} F_{\kappa}(e) = val$ with $val=0$ for $v \neq s,t$, $val = D(s,t)$ for $v=t$ and $val = -D(s,t)$ for $v=s$. We say that $D(s,t)$ is the \emph{value}  $|F_{\kappa}|$ of $F_{\kappa}$. Lastly mode assignment of each commodity is given by a \emph{modal split} $M_{\kappa}: E \to [0,1]^{m+1}$, which tells for each commodity, which percentage of the travelers use mode $i$ on edge $e$, where we require $\sum_{i=0}^{m} M_{\kappa}(e)_i = 1$ and $\sum_{ \kappa \in K} F_{\kappa}(e) M_{\kappa}(e)_i \leq k_i L(e)_i$ for all $e \in E$ and $i=1,\ldots,m$. 

We mostly use notation $F(e), M(e)_i$ for the accumulated values over all commodities for simplification, \ie \[F(e) := \sum_{\kappa} F_{\kappa}(e)\text{ and }M(e)_i := \frac{\sum_{\kappa}F_{\kappa}(e) M_{\kappa}(e)_i}{F(e)}\] for $i=0,\ldots,m$ if $F(e) > 0$ and $M(e) = (1,0,\ldots,0)$, else. A feasible \emph{solution} consists then of a triplet $(F,L,M)$, where $F = (F_{\kappa})_{\kappa}$ are flows for each commodity with values $|F_{\kappa}| = D(\kappa)$, $L$ is a feasible layout and $M$ a feasible modal split. Given a feasible solution $(F,L,M)$, the \emph{total travel time} evaluates as 

\begin{equation}
    \Tau(F,L,M) = \sum_e w(e)  \sum_i \tau_i F(e) M(e)_i
\end{equation}

and the \emph{total energy consumption} is

\begin{equation}
    \Eta(F,L,M) = \sum_e  w(e) \left(\eta_0F(e)M(e)_0 + \sum_{i > 0} \eta_i L(e)_i\right).
\end{equation}

The goal of our work is to compute proper solutions for the \ac{MSP} and to investigate the difficulty of this task. But the problem consists of two objective functions, which typically interfere, implying that there is in general no solution which is optimal in both objectives. A solution $\sol = (F,L,M)$ \emph{dominates} another $\sol' = (F',L',M')$ if $\Tau(\sol) \leq \Tau(\sol')$ and $\Eta(\sol) \leq \Eta(\sol')$, where strict inequality holds in at least one objective. In this case we also say that $(\Tau(\sol), \Eta(\sol))$ \emph{dominates} $(\Tau(\sol'), \Eta(\sol'))$. A goal of bicriteria optimization is often to compute a set of solutions, which are not dominated by any other solution. Such solutions are called \emph{efficient} or \emph{Pareto optimal}. The set of all efficient solutions is the \emph{efficient set} and its image the \emph{Pareto frontier}. A general introduction to multicriteria optimization (MCO) with rigorous definitions is found for example in \cite{ehrgott05}. One method to tackle MCO problems is the $\varepsilon$-constraints method, where typically all but one objective functions are bound and the problem is optimized for the remaining one. This gives rise to a definition of a decision problem for \ac{MSP}:

\begin{decisionproblem}{\ac{MSP}}
{A directed, weighted graph $G=(V,E,w)$ with $w: E \to \mathbb{R}_{\geq 0}$, a budget $B > 0$, demand $D: V^2 \to \mathbb{N}_{\geq 0}$ and mode data $\tau_i, \eta_i, c_i, k_i$ for $i=0,\ldots,m$ with $c_0=0$ and $k_0=1$ and $a,b \geq 0$.}
{Does there exist a feasible solution $(F,L,M)$ with \\$\Tau(F,L,M) \leq a$ and $\Eta(F,L,M) \leq b$?}
\end{decisionproblem}

%% file: complexity.tex
As already described in the introduction, the \ac{MSP} is closely related to classical network design, which will be the baseline for reduction. The complexity of many variants of network design problems have been investigated various times, but they often lack the directed counterparts. Therefore we give proofs and translations for the directed variant for existing theorems, and use those results to reduce to the \ac{MSP}. As this section is notation-heavy, the common used symbols and nomenclature throughout the paper are again listed at the end of the document for the readers convenience.

In directed network design we are given a directed graph $G = (V,E)$ with both a weight and a distance function $w,d: E \to \mathbb{R}_{\geq 0}$ on the edges. A directed path from node $u$ to node $v$ is a sequence of nodes $u=v_1, \ldots, v_k = v$ such that $(v_i, v_{i+1}) \in E$ for all $i=1,\ldots,k-1$. We call $\sum_{i=1}^{k-1} d(v_i,v_{i+1})$ the length of the path, which gives rise to a distance measure between nodes by $\dist(u,v)$ being the length of a shortest, directed path \wrt $d$, where we set $\dist(u,v) = \infty$ if no directed path from $u$ to $v$ exists. The \emph{total routing costs} of $G$ is the sum over all pairwise distances in $G$, \ie $R(G) := \sum_{u,v} \dist(u,v)$. When considering subgraphs induced by subsets $E'$ of $E$, we consider by $\dist|_{G'}$ the natural restriction of $\dist$ from $G$ to $G'=(V,E')$, analogously we may write $w(E')$ for $\sum_{e \in E'} w(e)$ for the total weight of $E'$. The decision problem \ac{DiNDP} is then defined as follows:

\begin{decisionproblem}{\ac{DiNDP}}
{A directed, weighted graph $G=(V,E,w,d)$ with $w,d: E \to \mathbb{Z}_{\geq 0}$ and $\budgetB, \budgetR \in \mathbb{Z}_{\geq 0}$.}
{Does there exist a subset $E' \subset E$ such that $w(E') \leq \budgetB$ and $R(G'=(V,E',w,d)) := \sum_{u,v \in V} \dist|_{G'}(u,v) \leq \budgetR$?}
\end{decisionproblem}

We will show \NP-completeness of \ac{DiNDP}, even in the unweighted case, by a reduction from \ac{X3C}, which will then be used to show \NP-completeness of \ac{MSP}. This implies \NP-completeness of the general case, but both the unweighted and the general variant of \ac{DiNDP} provide different approximation guarantees, therefore we also provide the results for the general case as well. The proofs for the general cases are to be found in the Appendix as \Cref{lem:appDiNDPComplete} and \Cref{lem:appDiNDPApprox}. All \NP-hardness results for the undirected network design has been established in \cite{LenstraNDP78}; the proofs for the directed versions follow similar ideas, but need some effort.

\begin{decisionproblem}{\ac{X3C}}
{A finite set $M = \{1,\ldots,n\}$ with $n \pmod 3 \equiv 0$ and a finite subset $\mathcal{S} \subset 2^M$ with $|S| = 3$ for all $S \in \mathcal{S}$.}
{Does there exist a subset $\mathcal{S}' \subset \mathcal{S}$ with $\cup_{S \in \mathcal{S}'} S = M$ and \\ $S \cap \tilde{S} = \emptyset$ for all $S,\tilde{S} \in  \mathcal{S}'$?}
\end{decisionproblem}

\ac{X3C} is \NP-complete, see \eg \cite{Karp1972}.

\begin{lemma} ~
\ac{DiNDP} is \NP-complete even if $w(e) = d(e) = 1$ for all $e \in E$ and $ \budgetB= 2(|V|-1).$
\end{lemma}

\begin{proof} By reduction from \ac{X3C}, let an instance $I=(M, \mathcal{S})$ be given with $n := |M|$ and $n=3l$ for some $l$, $|\mathcal{S}| = k\geq l$, and define an instance $J$ of \ac{DiNDP} by constructing a graph as in \Cref{fig:regulartree4}, where $h,\budgetR$ will be specified below. Connect $w_i$ and $v_j$ by a symmetric edge if $i \in S_j$. 

\begin{figure}[ht] 
\begin{center}
\includegraphics{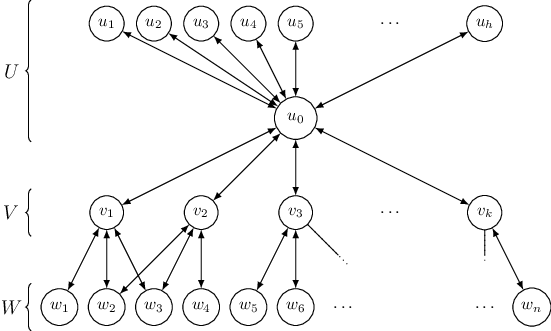}
\caption{The \ac{DiNDP} instance used for reduction. The bottom nodes are connected to some $v_j$ if and only if the corresponding elements are contained in the subset $S_j$.}
\label{fig:regulartree4}
\end{center}
\end{figure}

\begin{table}[h]
    \centering
    \begin{tabular}{|c|c|c|c|}
        \hline
        $C_{U,U}^*$ & $2h^2$ & $C_{U,W}^* = C_{W,U}^*$ & $3hn+2n$  \\
        \hline
         $C_{V,V}^*$  & $2k(k-1)$   & $C_{V,W} = C_{W,V}$ & $9nk-2n$ \\
        \hline
        $C_{W,W}$ & $4n^2-8n$ & $C_{U,V}^* = C_{V,U}^*$ & $2hk+k$   \\
        \hline
    \end{tabular}
    \caption{The routing costs between the subsets $U,V$ and $W$ on a variant of graph $G$ where exactly three bottom nodes are connected to some intermediate node or none.}
    \label{tab:routingcostsoptimal}
\end{table}

\Cref{tab:routingcostsoptimal} shows the routing costs for all pairs of the node sets $U,V,W$ for a similar graph $(V,E^*)$,  that is connected, symmetric and each node $v_i$ is connected to either exactly three nodes $w_j$ or to none and each node $w_j$ is connected to exactly one node $v_i$.

All entries with an asterisk are lower bounds on the routing costs for all spanning subgraphs of $G$, as one quickly verifies by the node distances affected. Set $h := C_{V,V}^* + 2C_{V,W} + C_{W,W}$ and $\budgetR$ to be the sum of all entries in \Cref{tab:routingcostsoptimal} (counting $C_{A,B}$ with $A \neq B$ twice). Assume $J$ decides true and let $E'$ be a spanning subgraph of size $2(|U \cup V \cup W|-1)$ and routing costs at most $\gamma$. We show that without restriction we can assume the resulting graph to be symmetric, the claim then follows by the undirected result. The edges in between $U$ which are contained in $E'$ are clearly symmetric, otherwise some $u_i$ is not reachable or vice versa. Analogously, the edges between $U$ and $V$ are symmetric as well, and there is exactly one symmetric edge between $u_0$ and each $v_j$. Assume the contrary, \ie that $(u_0,v_j) \notin E'$ for some $j$. Then the distance from each $u_i$ to $v_j$ increases by at least $2$ compared to the optimal case of $C^*_{U,V}$, leading to

\begin{align*}
    R(G') &\geq R_{U,U}(G') + R_{U,V}(G') + R_{V,U}(G') + R_{U,W}(G') + R_{W,U}(G')\\
        &\geq C_{U,U}^* + (C_{U,V}^* + 2h)+ C_{V,U}^* +2C_{U,W}^* \\
        &= \budgetR + h > \budgetR.
\end{align*}

Analogously, if $(v_j,u_0) \notin E'$, then $R_{V,U}(G') \geq (C_{V,U}^* + 2m)$, leading to the same result, thus $(u_0,v_j), (v_j,u_0) \in E'$. Lastly, we show symmetry of the edges in $E'$ between $V$ and $W$. Assume there is some $(v,w) \in E'$, with $(w,v) \notin E'$. But then $(w,v') \in E'$ due to connectedness for some $v' \in V$. There might be $r_1 \leq 2$ additional nodes $w_1,w_2$ connected to $v$ and $r_2 \leq 2$ many nodes $w_1',w_2'$ to $v'$, see \Cref{fig:symmetrizationtree} for the situation.

\begin{figure}[ht] 
\begin{center}
\includegraphics{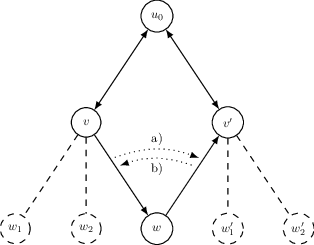}
\caption{Illustration of the two operations to create a subgraph with symmetric edges between $V$ and $W$. Operation a) deletes edge $(v,u)$ and adds edge $(u,v')$, whereas b) does the opposite. Both operations only affect the distances between the displayed nodes, if existing.}
\label{fig:symmetrizationtree}
\end{center}
\end{figure}

We compute the change in routing costs when tilting either $(v,w)$ to $(v',w)$ as option a) or $(w,v')$ to $(w,v)$ as b).  For a) the only distances affected are the ones from all nodes except $u_0$ to $w$ in \Cref{fig:symmetrizationtree}. The distances from $v,w_1,w_2$ increase by $2$ whereas the ones from $v',w_1',w_2'$ to $w$ decrease by $2$, which leads to additional routing costs of  $2(r_1-r_2)$. Conversely for option b) the routing costs increase by $2(r_2-r_1)$, which implies that either a) or b) don't increase the routing costs. Hence with the prior results, we can compute a symmetric subgraph in polynomial time that spans the nodes with at most the same routing costs. The claim follows now by the undirected variant (see \cite{LenstraNDP78}), where it is shown that $E'$ has routing costs of at most $\gamma$ if and only if it has the structure of $E^*$, \ie if each $v_i$ is symmetrically connected to exactly three bottom nodes $w_j$ or none. In this case the partition into sets of three is given by the connections from $V$ to $W$.
\end{proof}

\begin{theorem} \label{thm:MSPHard}
\ac{MSP} is \NP-complete, even if $w(e) = 1$ for $e \in E$, $m=1$, $D(u,v)=1$ for $u \neq v$ and $\eta_1, \tau_1 = 1$.
\end{theorem}

\begin{proof}
By reduction from \ac{DiNDP}, let $I := ((G,w=d=1),\budgetB,\budgetR)$ be an instance and w.l.o.g. assume $G$ is strongly connected, otherwise \ac{DiNDP} decides false. Define an instance $J$ of \ac{MSP} by copying $(G,w)$, and set $D(u,u) = 0$ for all $u$ and $D(u,v) = 1$ for $u \neq v$. We set $m=1$, thus besides individual transport we have one public transportation mode and set $\eta_1 = 1, \eta_0 = |E|(\budgetB+1)$, $k_1 = |E|^2$, $\tau_1 = 1$ and $\tau_0 = 0.25$. Finally we set $(a,b) := (\budgetB,\budgetR)$.

Assume \ac{MSP} decides true on $J$, let $(F,L,M)$ be a solution and set $E' := \{e \mid L(e)_1 > 0\}$ to be the subnetwork of mode $1$. Then $(V,E')$ is connected; assume the contrary, then there exists a cut $C$ in $G$ separating $s$ and $t$ for some $s \neq t$ that consists only of edges with $L(e)_1 = 0$. As $C$ is a cut and $F_{(s,t)}$ a feasible flow, it follows $\sum_{e \in C} F_{(s,t)}(e) \geq D(s,t) = 1$. But then $\Eta(F,L,M) \geq  \eta_{0} \sum_{e \in C} F_{(s,t)}(e) \geq \eta_{0}|F_{(s,t)}|  \geq (\budgetB+1) |F_{(s,t)}| \geq (\budgetB+1) > \budgetB,$ a contradiction.

Now construct another solution $(F',L',M')$ by setting $L' = L$ and rerouting the passengers such that mode $0$ is unused as follows: For any edge let $F_e^0 := M(e)_0 F(e)$ be the total number of mode-$0$ passengers on $e = (u,v)$ in $F$. As $E'$ is connected there is a shortest path in $u$ to $v$ in $E'$ with length at most $|E'| \leq |E|$. Reroute all mode-$0$ passengers on $e$ over that path and update $M$ accordingly, such that they use mode $1$ in $E'$. This reduces the total energy consumption and increases the travel time by at most $F_e^0 |E| \tau_{1} = F_e^0 |E|$. Doing so for every edge increases the travel time by at most $\sum_e F_e^0 |E|$ and delivers a solution with no mode-$0$ passengers at all. Lastly reroute all passengers in $E'$ on the shortest paths in $E'$ in mode $1$ to further reduce the travel time. Note that the capacity of mode $1$ is large enough to keep any amount of passengers.

\begin{figure}[ht] 
\begin{center}
\includegraphics{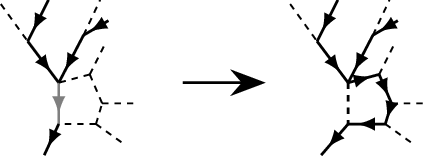}
\caption{Rerouting the flow from mode $0$ (in gray) on $e$ over a path of mode $1$ of the dashed edges.}
\label{fig:flowrerouting}
\end{center}
\end{figure}

Consequently, we have $\Eta(F',L',M')  = \sum_e L'(e)_1 \eta_{1} \leq \sum_e M(e)_0 F(e)\eta_{0} + L(e)_1\eta_{1}) = \Eta(F,L,M) \leq a = \budgetB$. As $\budgetB \geq \sum_{e}F_e^0 \eta_{0} \geq  \sum_e F_e^0 |E|(\budgetB+1),$ it follows $\sum_e F_e^0  |E| \tau_{1} = \sum_e F_e^0 |E|  < 1$. But then $\Tau(F',L',M') \leq \Tau(F,L,M) + \sum_e F_e^0 |E| \tau_{1} < \Tau(F,L,M) + 1 \leq \budgetR + 1$ and therefore $\Tau(F',L',M') \leq \budgetR$ as $\Tau(F',L',M')$ and $\gamma$ are integral. Now setting $E' := \{e \in E \mid L'(e) > 0\}$ delivers a solution for the \ac{DiNDP} instance $I$, as the construction costs are bound by $\Eta(F',L',M')$ and the routing costs agree with $\Tau(F',L',M')$.

Assume \ac{DiNDP} decides true on $I$ and let $E'$ be a solution. Setting $L(e')_1 = 1$ and $M_{\kappa}(e') = (0,1)$ for $e' \in E'$ and $(1,0)$ else, and $F_{\kappa}$ to a shortest path flow in $(V,E')$ for all $\kappa \in K$ delivers a feasible solution with $\Eta(F,L,M) = \sum_{e \in E'}1 \leq \beta = a$ and 
$\Tau(F,L,M) = \sum_{e \in E'} F(e) 
= \sum_{s, t} \dist|_{G'}(s, t) \leq \budgetR = b$.
\end{proof}

Next we give an inapproximability result for \ac{MSP}. For this we make use of the directed Steiner tree problem, which is similar to its undirected variant. In the \ac{STP}, we are given an undirected, weighted graph $G=(V,E,w)$ with $w: E \to \mathbb{N}_{\geq 0}$, $\terms \subset V$ and $\budgetB \in \mathbb{N}_{>0}$ and ask for a tree $G$ of weight at most $\budgetB$ that spans $\terms$. \ac{STP} is a widely discussed problem, whose \NP-completeness has been established various times, even for the case $w(e) = 1$ for all $e \in E$, see \cite{garey1979computers} for example. The counterpart to a tree in the directed case is the arborescence (see e.g. \cite{korte12}).

\begin{definition}[Arborescence \& Anti-Arborescence] Let $G = (V,E)$ be directed. An \emph{arborescence} is a subset $A \subset E$ such that
\begin{enumerate}
\item $A$ contains no undirected cycle,
\item Every vertex $v \in V$ has at most one entering arc in $A$,
\item $A$ is weakly connected.
\end{enumerate}
If instead of (2) it holds that every vertex has at most one exiting arc in $A$, then $A$ is called an \emph{anti-arborescence}. 
\end{definition}

The main difference to the undirected case is that the arborescences don't guarantee directed paths between each pair of nodes connected, but only from a defined root to all other nodes: One can show that there exists a defined node $r$ such that there is a unique, directed path to each node connected by $A$, the \emph{root} of $A$. An arborescence/anti-arborescence is \emph{spanning}, if it strongly connects every node to $r$.

\begin{decisionproblem}{\ac{DiSTP}}
{A directed, weighted graph $G=(V,E,w)$ with $w: E \to \mathbb{N}_{>0}$, $r \in V, \terms \subset V$ and $B \in \mathbb{N}_{\geq 0}$.}
{Does there exist an arborescence $A$ in $G$ of weight at most $B$, rooted at $r$ such that there exists a directed path in $A$ from $r$ to $v$ for all $v \in \terms$?}
\end{decisionproblem}

By a straightforward reduction from \ac{STP} one sees that \ac{DiSTP} is \NP-complete, even if $r \in \terms$ and $w(e) = 1 \ \forall e \in E$. We make use of this result to show inapproximability of the optimization problem \ac{MSP}, where we aim for minimizing $\Eta$.

\begin{searchproblem}{(\textbf{MSP})}
{A directed, weighted graph $G=(V,E,w)$ with $w: E \to \mathbb{N}_{>0}$, a budget $B$, demand $D$ and mode data $\tau_i, \eta_i, c_i, k_i$ for $i=0,\ldots,m$ with $c_0=0$ and $k_0=1$ and $a \geq 0$.}
{Find a solution $(F,L,M)$ with $\Tau(F,L,M) \leq a$ and $\Eta(F,L,M)$ minimal.}
\end{searchproblem}

A solution is an \emph{$\alpha$-approximation} for some instance of (\textbf{MSP}), if $\Tau(\sol) \leq a$ and $\Eta(\sol) \leq \alpha \Eta^*$, where $\Eta^*$ is the minimal value of $\Eta$ amongst all solutions with $\Tau(\sol) \leq a$. We say that (\textbf{MSP}) is $\alpha$-approximable, if there exists an algorithm that runs polynomial in $n$ and computes for each instance an $\alpha$-approximation, where $\alpha$ may be a constant or even a polynomial function in the parameters. The notation $a=\infty$ simply means that $\Tau$ does not need to be bounded to achieve inapproximability.

\begin{theorem} \label{thm:inapproxMSP}
    For any polynomial time computable function $\alpha(|V|)$, (\textbf{MSP}) is not $\alpha(|V|)$-approximable, unless $\textbf{P} = \NP$, even if $m=1$, $w(e)=d(e)=1$ for all $e \in E$ and $a = \infty$.
\end{theorem}

\begin{proof}
    Let $I = (G =(V,E,w=1),r,\terms,\budgetB)$ be a \ac{DiSTP} instance with $|V| = n$, $\beta \leq |E|$,  $w(e) =1 \ \forall e \in E$, copy $G$ to define an instance $J$ of (\textbf{MSP}). Set $D(r,v) = \lceil \alpha(n)(|E|+1) \rceil$  for $v \in \terms\setminus{r}$. Lastly we set $B = \budgetB$ and have $m=1$ public transportation mode with $k_1 = \sum_{u,v} D(u,v)$, $c_1=1$ and $\eta_0 = \eta_1= 1$. Then $J$ is an instance, which is polynomial in size in all parameters of $I$, especially in $n$. Note that with those parameters, $|E|$ is an upper bound for $\Eta$, in the case that an arborescence of costs at most $B$ exists. 

    \medskip

    Assume there is a polynomial time algorithm to compute a solution $\sol$ for $J$ with $\Eta(\sol) \leq \alpha(n) \Eta^*$. If $I$ has a solution $A$, we define a solution $\sol'$ by setting $A$ to be the public transport network of mode $1$ and routing all passengers through it, leading to $\Eta^* \leq \Eta(\sol') \leq |E|$, thus $\Eta(\sol) \leq   \alpha(n) \Eta^* \leq \alpha(n) |E|$. If $I$ is not solvable, also no public network of costs at most $B$ that connects each $v \in \terms$ to $r$ oneway exists, and therefore at least one commodity $(r,v)$ traverses an edge with mode $0$ in every solution, leading to $\Eta(\sol) \geq \alpha(n)(|E|+1) > \alpha(n) |E|$. Thus \ac{DiSTP} decides true on $I$ if and only if $\Eta(\sol) \leq \alpha(n) |E|$, hence would be decidable in polynomial time.
\end{proof}

Analogously the optimization problem (\textbf{DiNDP}) asks for a subgraph $(V,E')$ with $\sum_{e \in E'} w(e) \leq \beta$ such that $R(G')$ is minimal. Inapproximability within $|V|^{(1-\varepsilon)}$ for the undirected \ac{NDP} has been established in \cite{wong80}. The proof for the directed case follows the same ideas and is contained in the Appendix as \Cref{lem:appDiNDPApprox}.

Also contained in \cite{wong80} is a $2$-approximation for the case of the undirected network design problem, where the edge weights $w(e)$ agree to be $1$. We modify this result for the directed case and give an approximation for a special case of (\textbf{MSP}).  Analogously to shortest path spanning trees, shortest path arborescences can be computed in polynomial time, an algorithm can be found in \cite{wang2015efficient}, \ie in polynomial time one can compute an arborescence $A^+_v$, rooted at $v$ that spans $V$, such that the unique, directed path from $v$ to $w \in V$ is a shortest path in $G$. Correspondingly we can do the same to achieve an anti-arborescence $A^-_v$, then $G_v := (V,E_v := A^+_v \cup A^-_v)$ is a strongly connected subgraph of $G$, a \emph{shortest path subgraph w.r.t. $v$}. Note that analogous to the undirected case, the number of edges of a spanning arborescence is exactly $|V|-1$.  Choosing a central node as $v$ and combining both arborescences delivers a $2$-approximation for (\textbf{DiNDP}) in the case of $w(e)=1$. We denote by $R^+_v$ the routing costs from $v$ to all $w \in V$ and by $R^-_v$ from all $w \in V$ to $v$.

\begin{lemma} \label{lem:2Approx}Let $G$ be directed and connected. Let 
\[
v := \argmin \left\{R_w^+ + R_w^- \mid w \in V\right\}
\]

and $E_v := A^+_v \cup A^-_v$. Then $R(G_v) \leq 2 \cdot R(G)$.
\end{lemma}

\begin{proof} We have $R(G) = \sum_{w \in V} R^+_w$ and $R(G) = \sum_{w \in V} R^-_w $, thus $R(G) = \frac{1}{2} \sum_{w \in V} (R^+_w + R^-_w) \geq \frac{1}{2} \sum_{w \in V} (R^+_v + R^-_v) = \frac{n}{2}(R^+_v + R^-_v)$.  On the other side it holds $R(G_v) = \sum_w R_w^+(G_v) \leq \sum_w \sum_u (\dist|_{G_v}(w,v) + \dist|_{G_v}(v,u)) = \sum_w \sum_u (\dist(w,v) + \dist(v,u)) = n (R^-_v +  R^+_v)$. Thus  $R(G_v) \leq n(R_v^-+R_v^+) \leq 2 \cdot R(G)$.
\end{proof}

As a direct consequence we achieve a $2$-approximation for (\textbf{DiNDP}) in the case that the construction costs $w$ for each edge agree to $1$. Additionally \Cref{lem:2Approx} delivers a solution for special instances of \ac{MSP}, which approximates the $\Eta$-Pareto extreme in both objective functions.

\begin{corollary} Let $I$ be an instance of \ac{DiNDP} with $w(e) = 1$ for all $e \in E$ and $\budgetB \geq 2(|V|-1)$. Then $E_v$ defined as in \Cref{lem:2Approx} is a $2$-approximation for $I$.
\end{corollary}

\begin{corollary} Let $I$ be an instance of (\textbf{MSP}) with $m=1$, $k_1 \geq (|V|-1)^2$, $c_1 \leq 1$, $B \geq |E|$, $w(e) = 1 \forall e$ and $\eta_0 > \eta_1$ and $D(s,t) = 1 \forall s \neq t$. Then there is a solution $\sol$ which can be computed in polynomial time, such that $\Eta(\sol) \leq 2 \Eta^*$ and $\Tau(\sol) \leq 2 \Tau^{\Eta^*}$, where $\Tau^{\Eta^*}$ is the minimal travel time over all solutions with an total energy consumption of $\Eta^*$.
\end{corollary}

\begin{proof} Let $E_v := A^+_v \cup A^-_v$ as in \Cref{lem:2Approx} and $L(e) = 1$ if $e \in E_v$; furthermore let $F$ be a shortest path flow in $G_v$ and $M_e = (0,1)$ for all $e \in E_v$. $(F,L,M)$ is feasible due to $B$ and $k_1$ being sufficiently large. Let $\sol^*$ be a solution with $\Eta(\sol^*) = \Eta^*$ and $\Tau(\sol^*) = \Tau^{\Eta^*}$ and $L^*$ its layout. Since $\Eta^*$ is minimal, $\eta_0 > \eta_1$ and $B, k_1$ sufficiently large, there is a connected mode-$1$ sub network in $L^*$, with no mode-$0$ passengers in $\sol^*$ at all. But this implies $\Tau(\sol^*) \geq \tau_1 R(G)$. But then Lemma \ref{lem:2Approx} leads to $\Tau(F,L,M) = \tau_1 R(G_v) \leq \tau_1 2 R(G) \leq 2 \Tau(\sol^*)$. Since $\sol$ consists only of a mode-$1$ network with no mode-$0$ passengers at all, it holds $\Eta^* = \sum_e L(e) \eta_1 = \sum_{e: L(e) = 1} \eta_1$ and since $L^*$ is spanning and $w(e) = 1$, it follows $\Eta^* \geq \eta_1(|V|-1)$, as any spanning subgraph needs to have at least $(|V|-1)$ edges. Lastly $E_v$ has at most $2(|V|-1)$ edges and therefore $\Eta(\sol) \leq 2 \eta_1(|V|-1) \leq 2\Eta^*$.
\end{proof}

%% file: frontier.tex
Our last results show hardness of \ac{MSP} even in the easiest cases, but also their approximability. In network design, the main difficulty comes from the interplay between choosing routes for each commodity, and simultaneously providing infrastructure to enable those routes. If the passenger routes, \ie the flow, is set to be fixed, the infrastructure demand follows immediately. Due to the choice of multiple modes and capacity bounds on the vehicles, this task is a challenge by itself in the \ac{MSP}, even in the case that the passenger flow is predetermined or uniquely defined in advance due to a tree-like structure of the graph. In this context we call a directed graph \emph{tree-like}, if it is connected and does not contain non-trivial, directed cycles. It is easy to see that $G$ is tree-like, if and only if there exists exactly at most one cycle-free, directed path from $s$ to $t$ for every pair $(s,t)\in V^2$. We show now that even in those cases \ac{MSP} remains hard. The difficulty arises from a Knapsack like structure of the resulting problem, where the arcs have to be chosen for installment of public transportation devices.

\begin{decisionproblem}{\ac{MKPS} \& \ac{UKPS}}
{A set $S$ of $n$ items, each defined by a tupel $i := (s_i, w_{i,1}, \ldots, w_{i,r})$, where $s_i$ is its value and $w_{i,j}$ are the weights, together with $r$ bounds $A'_j$ and a target value $A$.}
{Do there exist integers $x_i$ such that $\sum_{i \in S} x_i s_i \geq A$ and \\ $\sum_{i \in S} x_i w_{i,j} \leq A'_j$ for $j=1,\ldots,r$? If $r=1$ then the problem is simply called \ac{KPS}; if $x_i \in \{0,1\}$ we call the problems \emph{bounded}. The bounded \ac{KPS} with $s_i = w_i$ is called \ac{SSUM}, if additionally $A=A'$, then it is called \ac{ESUM}.}
\end{decisionproblem}

\begin{theorem} \label{thm:MSPTreeNPComplete} ~
\begin{enumerate}[leftmargin=1.1cm, label=\alph*)]
\item \label{thm:msptreea} \ac{MSP} is \NP-complete even if $F$ is fixed to be any feasible flow. As a consequence, \ac{MSP} is \NP-complete, even if $G$ is tree-like. This also holds in the case $m=1$.
\item \label{thm:msptreeb} \ac{MSP} is \NP-complete, even if $G$ consists of only one directed edge $e$ with weight $w(e) = 1$.
\end{enumerate}
\end{theorem}

\begin{proof}
\begin{enumerate}[leftmargin=1.1cm, label=\alph*)]
\item Let $I:=(S,A,A')$ be an instance of \ac{SSUM} with $n = |S|$. Construct an instance $J$ of \ac{MSP} by setting $G$ to be a one-directed path graph with $|V| = n+1$, $w(e_i) = d(e_i) = s_i$ and $D(v_1, v_{n-1}) = 2$, see \Cref{fig:linegraph}. Then $G,D$ directly imply the only feasible flow $F$ with $F(e) = 2$ for all $e \in E$. Set $\eta = (1,1), c_1 = 1, k_1 = 2, B = A'$ and $b := \overline{S}-A$ where $\overline{S} := \sum_{i \in S} s_i$. $\tau$ and $a$ can be chosen arbitrarily. Let $L$ be a layout that solves $J$, without restriction we can assume $L(e) \in \{0,1\}$, $M_{\kappa}(e) = (0,1)$ if $L(e)_1 = 0$ and $M_{\kappa}(e) = (1,0)$, else.  This is due to the fact that $M$ maximizes the passengers of mode $1$, thus any other modal split at most increases the energy consumption $\Eta$ or is infeasible. The base consumption with all passengers using mode $0$ is $\eta_0 \overline{S} = \overline{S}$, each edge $e$ with $L(e) = 0$ reduces the energy consumption by $s_i$, thus $\Eta(F,L,M) = \overline{S} - \sum_{e_i: L(e_i) = 1} s_i \leq b = \overline{S} - A$, leading to $\sum_{e_i: L(e_i) = 0} s_i \geq A$. On the other hand the costs of $L$ are bounded by $A' = B \geq \sum_{e_i: L(e_i) =1} w(e_i)c_1 = \sum_{e_i} L(e_i)s_i$, thus $\{s_i \mid L(e_i) = 1\}$ is a solution for $I$. If $S' \subset S$ is a solution for $I$, then $(F,L,M)$ with $L(e_i) = 1$ if $s_i \in S'$ solves $J$ with similar considerations.

\begin{figure}[ht]
\begin{center}
\includegraphics{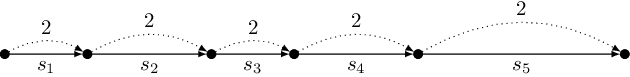}
\caption{The path graph $G$ with a flow value of $2$ between subsequent nodes and edge lengths corresponding to the items in $S$.}
\label{fig:linegraph}
\end{center}
\end{figure} 
\item By \ac{UKPS}, let $I:=(S,A,A')$ be an instance with w.l.o.g. $s, w \geq 1$ for all items in $S$ and set $G=(\{u,v\})$. For each item $(s, w) \in S$ define a mode $(\tau, \eta, k, c)$ by setting $\tau$ arbitrary, $\eta = s$, $k=2s$ and $c=w$. We set $D(u, v) = 2 s_{\max} A'$ and $B = A'$.  Lastly we set $\eta_0 = 1$, $a = \infty$ and $b = \max(0, 2 s_{\max} A'-A)$. Note that if $b$ is smaller or equal to zero, then $I$ decides false; this follows as at least twice as many items as $A'$ need to be chosen in this case to have a value of at least $A$, which is impossible due to $w \geq 1$ for all items. In this case $J$ also decides false, so we assume $b$ to be strictly positive.

The base energy with all passengers traveling by mode $0$ is $\eta_0 F_e^0 = 2 s_{\max} A'$. If $L$ is any layout with $\sum_i L(e)_i k_i = \sum_i L(e)_i 2s_i > D(u, v) = 2s_{\max}A'$, then also $\sum_i L(e)_i c_i = \sum_i L(e)_i w_i \geq \sum_i L(e)_i \frac{s_i}{s_{\max}} > A' = B$ as $w_i \geq 1$. Hence any feasible layout does not exceed $D(s,t)$ in its total public mode capacity. Therefore and because $\eta_i < \eta_0 k_i$, we can w.l.o.g. assume that every mode set in every solution is full. Thus setting one vehicle of mode $i$ and assigning $k_i=2s_i$ passengers to it reduces the base energy by $s_i$. 

Let a solution $(F,L,M)$ to $J$ be given.
Setting $x_i := L(e)_i$ gives then a solution to $I$ as $\sum_i x_i w_i = \sum_i L(e)_i c_i \leq B = A'$ and $\sum_i x_i s_i = \sum_i L(e)_i \eta_i = F_e^0 - \Eta(F,L,M) \geq 2 v_{\max} A' - b = A$. Conversely if $(x_i)_i$ solves $I$, then $L(e)_i := x_i$ solves $J$; all vehicles can assumed to be full as well with the same arguments and therefore $\Eta(F,L,M) = F_e^0 + \sum_{i > 0}L(e)_i(\eta_i - \eta_0 k_i) = 2 v_{\max} W + \sum_{i > 0}x_i(s_i-  2s_i) = 2 s  _{\max} W - \sum_{i > 0}x_i s_i \leq 2 v_{\max} A' - A = b$.
\end{enumerate}
\end{proof}

At this point we need to mention that neither \Cref{thm:MSPTreeNPComplete} \ref{thm:msptreea} implies \ref{thm:msptreeb} or vice versa, as the case of one edge with only one mode is clearly solvable in polynomial time. Additionally neither \ref{thm:msptreea} nor \ref{thm:msptreeb} imply \Cref{thm:MSPHard}, as we assume $w(e) = 1, D(u,v) = 1$ for $u \neq v$ and $m=1$ there, also leading to tractable cases on tree-like graphs. In contrast to \Cref{thm:MSPTreeNPComplete}, \ac{DiNDP} and \ac{DiSTP} are solvable in polynomial time in the case of a tree-like graph. In the case of \ac{DiSTP} there exists an unique arborescence connecting the terminal nodes, if any at all. Similarly for \ac{DiNDP} all unique, directed paths from every $v$ to $u$ must exist in $E'$ for solvable instances, making $E'$ unique up to unnecessary edges, which only increase the construction costs and don't provide to the routing costs.

Next we show that (\textbf{MSP}), in the case of a fixed flow, admits polynomial time approximation schemes (PTAS).   As a consequence we achieve that computing the most energy efficient solution in the case that $F$ is fixed provides an FPTAS. Both results rely on approximation schemes for \ac{KPS} and \ac{MKPS}, respectively. \ac{KPS} provides a fully polynomial time approximation scheme, see \cite{kellerer2004improved} for an overview and \cite{MAGAZINE1981270} for a concrete scheme with runtime $\mathcal{O}(n/\varepsilon)$. In contrast, for \ac{MKPS} there exists a PTAS, but neither a FPTAS (\cite{erlebach2002}) and not even an  EPTAS (\cite{kulik2010there}), unless \textbf{P} $=$ \NP.

\begin{theorem}\label{thm:approxfixedflow} ~

\begin{enumerate}[leftmargin=1.1cm, label=\alph*)]
    \item \label{thm:fptasa}(\textbf{MSP}) with fixed flow and $a = \infty$ provides an PTAS. 
    \item \label{thm:fptasb} If additionally $m=1$ holds, then there is actually a FPTAS. 
\end{enumerate}
\end{theorem}

\begin{proof} Let any $\varepsilon > 0$ be given and let $\eta := \min\{\eta_i \mid i=1,\ldots,m\}$. W.l.o.g. we can assume $\eta \neq \eta_0$, as otherwise (\textbf{MSP}) is solved trivially by the all-mode-$0$ solution. As one quickly verifies, $\check{\Eta} := \eta \sum_e w(e)F(e)$ is a lower bound for $\Eta$. Let $\Eta_0 := \eta_0 \sum_e w(e)F(e)$ and $\Eta'(\sol) := \Eta_0 - \Eta(\sol)$ and define $\delta \leq \varepsilon \frac{\check{\Eta}}{\Eta_0-\check{\Eta}}$ then $\delta \leq \varepsilon \frac{\Eta^*}{\Eta_0-\Eta^*}$ as well. Consequently, if $\sol$ is any solution with $\Eta'(\sol) \geq (1-\delta)\Eta'^*$ then also $\Eta(\sol) \leq \delta \Eta_0 + (1-\delta) \Eta^* = \delta(\Eta_0 - \Eta^*) + \Eta^* \leq (1+\varepsilon) \Eta^*$ (note that $\Eta'^* = \Eta_0 - \Eta^*$). Therefore, a PTAS/FPTAS for maximizing $\Eta'$ leads to the same result for (\textbf{MSP}) in the case of a fixed flow.

\begin{enumerate}[leftmargin=1.1cm, label=\alph*)]
    \item Setting a vehicle of mode $i$ on $e$ costs $w(e)c_i$ and reduces the energy consumption by $w(e)(k\eta_0 - \eta_i)$, where $k \leq k_i$ is the number of passengers, transferred from mode $0$ to this vehicle. For each combination of mode $i$ and edge $e$ we add items to $S$, where each item is a tuple of size $|E|+2$. The first entry is the value of the item, the second one corresponds to the costs and the third to $n+2$'th ones to the passenger load on the edges, for this we assume the edges are numbered from $j=1,\ldots,n$.
    \begin{enumerate}[label=\roman*)]
        \item For each $j=1,\ldots,|E|, i=1,\ldots,m$ the item \[(w(e_j)(k\eta_0 - \eta_i), w(e_j)c_i, 0,\ldots,0, k_i, 0 ,\ldots,0),\] where $k_i$ is entry $j+2$ corresponding to edge $e_j$, is added $F(e_j) \pmod{k_i}$ times.
        \item A vehicle of mode $i$ is valuable with at least $\eta_i/\eta_0$ passengers, as otherwise the energy reduction is negative. To take care of every partially filled vehicle, we add for each $i, j$ and potential rest of passengers $f \in [\lceil\eta_i/\eta_0\rceil, \ldots, \min(k_i, F(e_j))] \cap \mathbb{N}$ one item $(w(e_j)(f\eta_0 - \eta_i), w(e_j)c_i, 0,\ldots, 0,f,0,\ldots,0)$, where $f$ is again in the column corresponding to $e_j$.
    \end{enumerate}
    As bounds $W_j$ we have the budget $B$ for the second entries, and for the edge entries the flow value $F(e)$ as an upper bound to not place more passengers in public transport as available.
    \item The proof follows the same idea as \ref{thm:fptasa}, but reduces to definition of a \ac{KPS} instance. But we do not need to take care of the number of passengers in mode $1$ on each edge, which is known. Therefore we add for each edge the item $w(e)(k_1\eta_0-\eta_1, c_1)$ to $S$ for $v_e = F(e) \pmod{k_1}$ times. The potential rest of passengers is known to be $f_e := F(e) - v_ek_1 < k_1$ for $e \in E$. Certainly, we do not need to take care to not exceed the passenger flow and only add one item $(f_e\eta_0-\eta_1, c_1)$, in the case that $f_e \geq  c_1/c_0$. As a bound for the second entries we have the budget $B$.
    
\end{enumerate}

Without detailed elaboration it should be clear, that by maximization of the \ac{MKPS}/\ac{KPS} instances just defined we maximize the energy reduction $\Eta'$ by setting $L(e)_i$ to the number of occurrences of the combinations $e,i$ that in $S'$; the modal split is adapted to the flow value in mode $i$ on $e$, which is the sum of the entries corresponding to $e$, in a straightforward manner. With the discussion above and the existence of approximation schemes for \ac{MKPS} and \ac{KPS}, the result follows.
\end{proof}

The approximations of \Cref{thm:approxfixedflow} can be used to sample points on the Pareto frontier. To avoid notation overhead, we also call the bicriteria optimization problem (\textbf{MSP}) and talk about parameters and instances simultaneously, as it should be clear from the context what is meant. Each (\textbf{MSP}) instance can be modeled in a straightforward manner as a bicriteria mixed-integer linear optimization program (see \eg \cite{Wong84}). Continuous flow variables $x_{\kappa,e,i}$ describe how many passengers of commodity $\kappa$ use mode $i$ on edge $e$, whereas discrete design variables $y_{e,i}$ model the layout $L(e)_i$ for each edge and mode. Allowing continuous, nonnegative values for $y$ is an relaxation for (\textbf{MSP}), as the feasible set increases. To use this for a simple sampling of frontier points, we make use of the fact that a part of the Pareto frontier of the relaxed problem (\textbf{MSPR}), \eg where we allow $L:E \to \mathbb{R}_{\geq 0}^m$, has a very simple structure:

\begin{theorem} \label{thm:frontierrelaxed}
    Let $I$ be an instance of (\textbf{MSPR}) with $\tau_0 \leq \tau_i$ for $i=1,\ldots,m$ and $\frac{\tau_i k_i}{\eta_i} \leq \frac{\tau_{i+1} k_{i+1}}{\eta_{i+1}}$ for $i=0,\ldots,m-1$ and assume $c_m/k_m \leq c_i/k_i$ for all $i$ with $\frac{\tau_i k_i}{\eta_i} = \frac{\tau_m k_m}{\eta_m}$. Let $F$ be a shortest path flow and define $\Psi_0 := (\sum_{e} w(e)F(e)) (\tau_0,\eta_0)$ and $\Psi_1 := (\sum_{e} w(e)F(e)) (\tau_m,\eta_m/k_m)$. 
    
    Let $\delta = \min(B, \sum_e w(e)F(e)\frac{c_m}{k_m})/(\sum_e w(e)F(e)\frac{c_m}{k_m})$, then the line segment \[\Psi := \{\Psi(\lambda) \mid \lambda \in [0,1]\} := \{\Psi_0 + \lambda \delta (\Psi_1-\Psi_0) \mid \lambda \in [0,1]\}\]
    is a subset of the Pareto frontier of $I$.
\end{theorem}

\begin{figure}[ht]
\begin{center}
\includegraphics{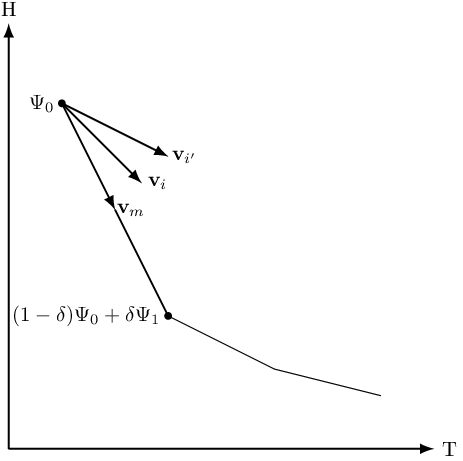}
\caption{The subsegment $\Psi$ of the Pareto frontier of the relaxation. The directional vectors $\vv_i := (\tau_i-\tau_0, \eta_i/k_i-\eta_0)$ show the change in both objective functions when passengers are transferred from mode $0$ to mode $i$. The vector $\vv_m$ of mode $m$ has the steepest slope and therefore provides directly to the Pareto frontier.}
\label{fig:paretoFrontierPsi}
\end{center}
\end{figure}

\begin{proof}

    We first show that $\Psi$ is contained in the image, i.e. corresponds to solutions. But this follows as $\Psi_0$ is precisely the image of the solution $(F,0,0)$ and $(F,L,M)$ with $L(e)_m = (0,\ldots,0,\delta (F(e)/k_m))$ and $M := (1-\delta,0,\ldots,0,\delta)$ is a preimage of $(1-\delta)\Psi_0 + \delta \Psi_1$. The intermediate points are convex combinations of $\Psi_0$ and $(1-\delta)\Psi_0 + \delta \Psi_1$ and therefore correspond to solutions due to convexity of the problem.
    Furthermore all solutions with image on $\Psi$ are pairwise incomparable, as $\Psi$ has slope of $-\frac{\tau_m k_m}{\eta_m}$, which is strictly negative and finite. Therefore for two solutions $\sol,\sol'$ with $(\Tau(\sol),\Eta(\sol)), (\Tau(\sol'),\Eta(\sol')) \in \Psi$ it follows either $\Tau(\sol) < \Tau(\sol')$ and $\Eta(\sol) > \Eta(\sol')$ or vice versa.
    
    \medskip
    
    Next we observe that every solution $\sol$ can be transformed into another one $\sol'$, where $F'=F$ is the shortest path flow, ${F'}_e^i = M'(e)_iF'(e) = k_i L(e)_i$ for $i=1,\ldots,m$ and $\Eta(\sol') \leq \Eta(\sol), \Tau(\sol') \leq \Tau(\sol)$. The latter is clear, as due to continuity of $L'$, there is no need of assigning more capacity than necessary. For the former let $\kappa = (s,t)$ be a commodity and assign the flow $F_{\kappa}$ to it. The percentage of mode $i$ passenger distance of $\kappa$ in $F'$ is $(\sum_e w(e) {M'}_{\kappa}(e)_iF'(e))/(\sum_e w(e) F'_{\kappa}(e))$. Setting $M_{\kappa}(e)_i$ to this value for all $e$ on the shortest path for all $i=0,\ldots,m$, and setting $L(e)_i = M(e)_iF(e)$ correspondingly, delivers a feasible solution as $(\sum_e w(e) F_{\kappa}(e)) \leq (\sum_e w(e) F'_{\kappa}(e))$. With the same argument $\Eta(\sol') \leq \Eta(\sol)$ and $\Tau(\sol') \leq \Tau(\sol)$.

    \medskip

    Now let $\sol$ be suchs a solution with shortest path flow and $L$-sharpness. We give a geometric argument, that $\sol$ can not dominate any solution on $\Psi$. Since $M(e)_0 = 1 - \sum_{i > 0} M(e)_i$, we can rewrite $(\Tau(\sol), \Eta(\sol))$ as
    
    \[ \sum_e w(e)F(e) \begin{pmatrix} \tau_0 \\ \eta_0 \end{pmatrix} + \sum_{i > 0} \left(\sum_e w(e)F(e) M(e)_i\right) \begin{pmatrix} \tau_i-\tau_0  \\ \eta_i/k_i-\tau_0\end{pmatrix} \] 
    
    since $L(e)_i = M(e)_iF(e)/k_i$ by sharpness. So from a geometric point of view, $(\Tau(\sol),\Eta(\sol))$ can be reached by starting at $\Psi_0$ and iteratively moving towards the directions $(\tau_i-\tau_0, \eta_i/k_i-\eta_0)$ one after another. The directional vectors $v_i := (\tau_i-\tau_0, \eta_i/k_i-\tau_0)$ have a strictly positive first component and strictly negative second component, with slope $\frac{\tau_i - \tau_0}{\eta_i/k_i - \eta_0} < 0$. But mode $m$ has the steepest slope, which is also the slope of $\Psi$, directed towards $\Psi_i-\Psi_0$ (see \Cref{fig:paretoFrontierPsi}). Certainly there is no way of moving below $\Psi$, or equivalently, there is no solution which strictly dominates any solution with objective value on $\Psi$.
\end{proof}

Any solution to an MCO problem, which is feasible and lies on the Pareto frontier of any relaxation is also Pareto optimal to the original problem. This is trivial, as dominating solutions to the original problem would also be dominating to the solution of the relaxation. Therefore, solutions that lie on $\Psi$ and are integral in $L$, are also Pareto optimal for the corresponding integer \ac{MSP} instance. As a consequence we achieve:

\begin{lemma} \label{lem:lowerparetooptimals} Let $I$ be an \ac{MSP} instance with the properties as in \Cref{thm:frontierrelaxed}. Then all solutions with $F$ being the shortest path flow, $L(e)_i = 0$ for $i=1,\ldots,m-1$ and $M(e)_mF(e) = k_m L(e)_m$ are Pareto-optimal.
\end{lemma}

\begin{proof}
    By rearranging $\Tau$ and $\Eta$ we see \[\begin{pmatrix}\Tau(F,L,M) \\ \Eta(F,L,M) \end{pmatrix} = (\sum_e w(e) F(e)) \begin{pmatrix}\tau_0 \\ \eta_0 \end{pmatrix} + (\sum_e w(e) k_mL(e)_m) \begin{pmatrix} \tau_m-\tau_0 \\  \eta_m/k_m-\eta_0\end{pmatrix}.\]
    It is left to show $\sum_e w(e) k_mL(e)_m \leq \delta(\sum_e w(e) F(e))$, but this follows since the budget spend, i.e. $c_m\sum_e w(e) L(e)_m$ is bounded by $\sum_e w(e)F(e) \frac{c_m}{k_m}$ as all vehicles set are full. But it is also bounded by $B$, thus $\sum_e w(e)k_mL(e)_m  \leq \frac{k_m}{c_m} \min(B, \sum_e w(e)F(e) \frac{c_m}{k_m}) = \frac{k_m}{c_m} \delta (\sum_e w(e)F(e) \frac{c_m}{k_m}) =  \delta \sum_e w(e)F(e)$. As $(\Tau(F,L,M), \Eta(F,L,M))$ lies on the Pareto frontier of the relaxation by \Cref{thm:frontierrelaxed}, $(F,L,M)$ is Pareto optimal itself.
\end{proof}

Using \Cref{lem:lowerparetooptimals}, we can sample layouts that are Pareto optimal. All samples provide solutions, where only full vehicles of mode $m$ are set. Hence the maximal budget for this algorithm is $\hat{B} := c_m\sum_e w(e)(F(e) \pmod{k_m})$. Now given $B \leq \hat{B}$ and an \ac{MSP} instance as in \Cref{thm:frontierrelaxed}, we can use the FPTAS of \Cref{thm:approxfixedflow} \ref{thm:fptasb} to compute a Pareto optimal solution that is an $\varepsilon$-approximation to all solutions with only modes $0$ and $m$ present, under the given budget.

\begin{algorithm} 
\caption{Lower Left Frontier Sampling}
\begin{algorithmic}
    \STATE \textbf{Input:} \ac{MSP} instance $I$, $B \leq \hat{B}$, $\varepsilon > 0$
    \STATE \textbf{Output:} Pareto optimal solution with budget spend at most $B$
    \STATE Define \ac{KPS} instance by $S : = \cup_e \{\underbrace{s_e := w(e)(k_1\eta_0-\eta_m, c_m)}_{v_e = F(e) \pmod{k_m} \text{times}}\}$ and $W := B$
    \STATE Compute $\varepsilon$-Approximation $S'$ by \ac{KPS}-FPTAS
    \STATE Return solution $(F,L,M)$, where $F$ is shortest path flow, $L(e)_m := \#(s_e \in S')$
    and $M(e)_i = k_iL(e)_i/F(e)$ if $F(e) > 0$, else $M(e) = (1,0,\ldots,0)$
\end{algorithmic}\label{alg:frontiersampling}
\end{algorithm}

Given any feasible layout $L$ for some \ac{MSP} instance $I$, then fixing $L$ delivers a continuous, bicriteria multi-commodity flow problem, which has a Pareto frontier itself. Such frontiers are called \emph{patches}, the Pareto frontier of $I$ is the Pareto frontier of the union of all its patches as a set. 

\begin{figure}[ht!]
\begin{center}
\includegraphics{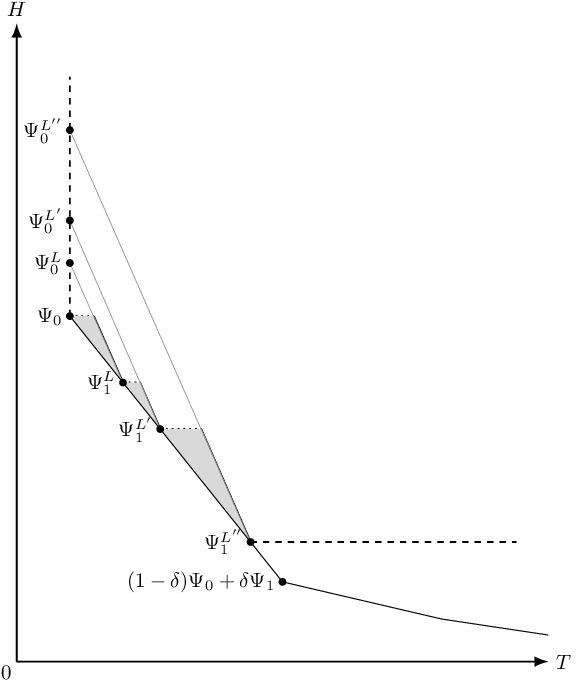}
\caption{Pareto optimal points $\Psi^L_1$ on $\Psi$ and their patches as grey lines. The patches are parallel and their dominated endpoints lie above $\Psi_0$. The shaded grey areas contain the Pareto frontier of all solutions, whose value lie in the area enclosed by the dashed lines and $\Psi.$}
\label{fig:paretoFrontierPsiPatches}
\end{center}
\end{figure}

In the context of \ac{MSP}, the patches desribe the users behavior under a given layout. They may use it or deny it, depending on the location of the patch points. We want to compute the patches for the solutions achieved with \Cref{alg:frontiersampling}, as they may contribute to the frontier of $I$. Let such a solution $\sol=(F,L,M)$ be given, \ie $F$ is a shortest path flow, $L$ only consists of mode $m$ and $L(e)_m k_m = F(e)M(e)_m$ for all $e \in E$. As $F$ and $L$ are fixed in the patch problem, the only degree of freedom left is $M(e)_m$ for all $e \in E$. 
Rearranging $\Tau(\sol)$ and $\Eta(\sol)$ and using $M(e)_0 = (1-M(e)_m)$ delivers for $(\Tau(\sol), \Eta(\sol))$ the value

\begin{equation} \label{eq:patchimage}
\begin{pmatrix} \tau_0\sum_e w(e) F(e) \\ \tau_0\sum_e w(e) F(e) + \eta_m L(e)_m \end{pmatrix} + (\sum_e w(e) F(e) M(e)_m) \begin{pmatrix} \tau_m - \tau_0 \\ - \eta_0\end{pmatrix}
\end{equation}

Additionally we observe $\left\{\sum_e w(e)F(e)M(e)_m \mid M(e)_m \in [0,1] \forall e \in E\right\} = \left\{\lambda \sum_e w(e)F(e) \mid \lambda \in [0,1]\right\}$: Relation $\supseteq$ is clear by choosing $M(e)_m = \lambda$, direction $\subseteq$ holds since $0 \leq \sum_e w(e)F(e)M(e)_m \leq \sum_e w(e)F(e)$ and $M(e)_m \leq 1$ for all $e \in E$. As a consequence, the image of the patch problem consists of the values of \Cref{eq:patchimage} for all feasible modal splits $M$ and is the line segment $\Psi^L$ from $\Psi^L_0 := (\tau_0\sum_e w(e) F(e), \tau_0\sum_e w(e) F(e) + \eta_m L(e)_m)$ and $\Psi^L_1 := \Psi^L_0 + (\sum_e w(e) F(e)) (\tau_m - \tau_0, - \eta_0)$. As the directional vector $(\tau_m - \tau_0, - \eta_0)$ has one positive and one negative entry, $\Psi^L$ is also the Pareto frontier of the patch problem as two different points in $\Psi^L$ are incomparable. The solution with value $\Psi^L_0$ describes a complete rejection of the layout, where every passenger travels by mode $0$ despite existence of $L$. For proper layouts $L \neq 0$, those solutions are strictly dominated by the all-mode-$0$ solutions without any existing layout, which have value $\Psi_0$. On the other hand, $\Psi^L_1$ solutions are fully accepted and Pareto optimal by \Cref{lem:lowerparetooptimals}. For the solutions on $\Psi^L$ in between the two endpoints no statement can be made in general, but they give an inner bound for the full Pareto frontier of $I$, whereas the segment $\Psi$ of the Pareto frontier of the relaxation serves as an outer approximation. See \Cref{fig:paretoFrontierPsiPatches} for a visualization.

%% file: acknowledgements.tex
This work was partially funded by the Bundesministerium für Bildung und Forschung (BMBF) in the context of the SynphOnie Project: Synergien aus physikalischen und verkehrsplanerischen Modellen zur multikriteriellen Optimierung multimodaler nachfrageorientierter Verkehre.

%% file: appendix.tex
\begin{lemma} \label{lem:appDiNDPComplete} \ac{DiNDP} is \NP-complete.
\end{lemma}

\begin{proof} ~
 \ac{DiNDP} is in \NP, as $R$ can be computed in polynomial time by the Floyd-Warshall algorithm (see \cite{floyd62}). Let $I = (S = \{s_1,\ldots,s_n\},A)$ be an instance of \ac{ESUM} and construct an instance $J = (G,\budgetB,\budgetR)$ of \ac{DiNDP} as follows:
Let $V := \{v_0\} \cup \{v_1,\ldots,v_n\} \cup \{v_1',\ldots,v_n'\}$ and set the edges as $E:= \cup_{i=1}^n \{(v_0,v_i), (v_i',v_0), (v_i,v_i'),(v_i',v_i)\}$ with $w(e)=d(e)=s_i$ if $v_i$ or $v_i'$ in $e$, where $e \in E$ (see  \Cref{fig:DiNDPESUMGraph} for the construction). 

\begin{figure}[ht]
\begin{center}
\includegraphics{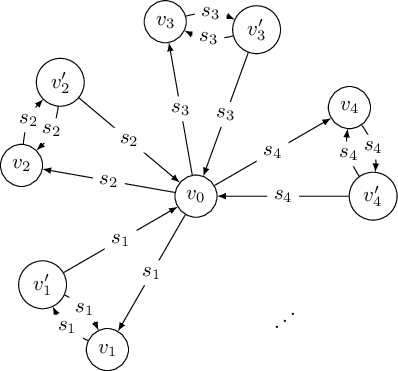}
\caption{The graph $G$ constructed for the reduction. Any connected subgraph needs to contain the cycles $v_0 \to v_i \to v_i'$. Adding the reverse edges $(v_i',v_i)$ then further reduce the routing costs by $s_i$.}
\label{fig:DiNDPESUMGraph}
\end{center}
\end{figure} 

Note that every connected subgraph $(V,E')$ of $G$ must contain the cycles $v_0 \to v_i \to v_i' \to v_0$. Let $G_c := (V,E_c)$ be the induced subgraph of those cycles and set $\overline{S} := \sum_{s \in S} s$ then $R(G_c)$ computes as $(12n-3)\overline{S}$ and $\sum_{e \in E_c} w(e) = 3\overline{S}$. Note that the only cycle-free paths in $G$ which make use of edges $(v_i',v_i)$ are those edges itself. Consequently, adding some $(v_i,v_i')$ to $E_c$ reduces only the distance from $v_i'$ to $v_i$ by $s_i$ and increases the total weight by $s_i$. Set $\budgetB := 3\overline{S} + A$ and $\budgetR := (12n-3)\overline{S} - A$ and assume $E' \subset E$ solves $J$. Then $E' = E_c \cup \{(v_i',v_i) \mid i \in \Gamma\}$ for some $\Gamma \subset \{1,\ldots,n\}$. Set $S':= \{s_i \mid i \in \Gamma\}$, then $\sum_{s \in S'} s= \sum_{i \in \Gamma} w(v_i',v_i) = \sum_{e \in E'} w(e) - \sum_{e \in E_c} w(e) \leq \budgetB - 3\overline{S} = A$ and $\sum_{s \in S'} s = R(G_c) - R(G') \geq (12n-3)\overline{S} - \budgetR = A$, thus $\sum_{s \in S'} s = A$ as desired.
If $S' = \{s_i \mid i \in \Gamma\}$ is a solution to $I$, then $E_c \cup \{(v_i,v_i') \mid i \in \Gamma\}$ solves $J$ with the same considerations.
\end{proof}

\begin{lemma}\label{lem:appDiNDPApprox} \ac{DiNDP} is not approximable within $|V|^{(1-\varepsilon)}$ for any $\varepsilon \in (0,1)$ unless \textbf{P} = \NP.
\end{lemma}

\begin{proof} By reduction from \ac{DiSTP}, let $(G,\terms, r\in \terms,B)$ be an instance with $h := |V| \geq 4$. Let  $k \geq 3$ such that $\frac{k-2}{k+2} \geq (1-\varepsilon)$. Define an instance of (\textbf{DiNDP}) by constructing a graph $\tilde{G}$ by copying $G$ and adding to each $v \in \terms$ a set of $h^k$ nodes $v_1,\ldots,v_{h^k}$ attached to $v$ two-ways with both values of $d$ and $w$ being $0$. Furthermore define an artificial node $q$ that is connected to all nodes in $V$, two-ways as well with $w$ equal to $0$ and routing costs $d$ of $2$. The edges in the copy of $G$ have the weight of the original graph and routing costs of $0$, the construction is shown in \Cref{fig:DiNDPApprox}. We set the demand to be $D(v_{1,j},v_{i,j'}) = 1$ for each star node $v_{1,j}$ at $r$ to each star node $v_{i,j}$ at $v_i$ for each terminal $v_i \neq v_1 = r$. Lastly let $n$ be the number of nodes in the graph $\tilde{G}$ just designed.

\begin{figure}[ht]
\begin{center}
\includegraphics{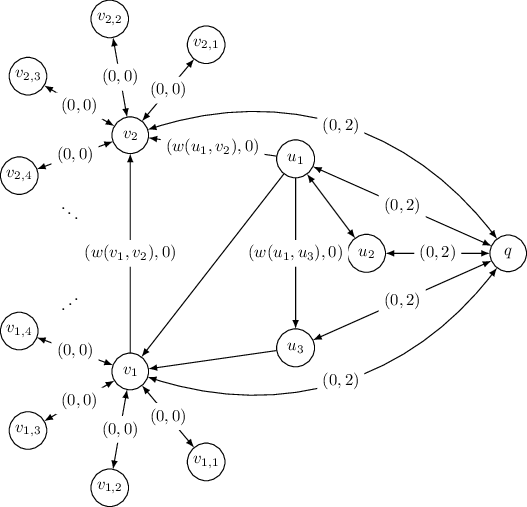}
\caption{The graph $\tilde{G}$; nodes $v_{i,l}$ are attached to each terminal node $v_i$ for $l=1,\ldots,h^k$ and a central node $q$ connects two-ways to all nodes, including the Steiner nodes $u_i$, in original $G$. The pairs denote the weight and routing costs $(w(e), d(e))$ of each edge.}
\label{fig:DiNDPApprox}
\end{center}
\end{figure}

It holds $n^{(1-\varepsilon)} \leq n^{\frac{k-2}{k+2}} \leq ( h^{k+1} +1)^{\frac{k-2}{k+2}} < (h^{k+2})^{\frac{k-2}{k+2}} = h^{k-2}$, since the number of nodes in $\tilde{G}$ is bounded by $(h-s)h^k  + s + 1 \leq h^{k+1}+1$, where $s = |V| - |\terms|$ is the number of Steiner nodes in $G$. Let $E' \subset \tilde{E}$ be a set of directed edges with $R(E') \leq n^{(1-\varepsilon)} R^* < h^{k-2}R^*$ and $\sum_{e \in E'} w(e) \leq B$. We show $R(E') < 4h^{2k} \Longleftrightarrow I$ decides true. 

Assume $R(E') < 4h^{2k}$, then for each $v \in \terms$ there exists a directed path from $r$ to $v$ in $E'$, without using any edges to $q$. Assume the contrary, and let such a node $v$ be given. Following the path over $q$ adds routing costs of at least $4$ for each pair of nodes attached to the stars of $r$ and $v$, respectively, which are $h^{2k}$ many, therefore leading to routing costs of at least $4h^{2k}$. By collecting paths from $r$ to all terminal nodes and choosing any arborescence rooted at $r$ in the resulting set of edges, we achieve a solution for $I$. 
Conversely, assume $I$ decides true and let $A$ be a corresponding arborescence in $G$ that spans $\terms$ and is rooted at $r$. Setting $E'' := A \cup \{e \in \tilde{E} \mid e \notin E\}$ delivers a subset with construction costs of $w(E'')= w(A) \leq B$ and routing costs $R(E'') := \sum_{u,v \in \tilde{V}} D(u,v) \dist|_{G''}(u,v) = 2 \sum_{v \in V} \dist(s,v) \leq 2h^{k+2}$ as all routes from $r$ so the stars are taken in the arborescence without routing costs and all routes from $q$ must be taken over the artificial edges as there is no other option. But then $R(E') < 2h^{k+2}h^{k-2} = 2h^{2k} < 4h^{2k}$. Consequently, $I$ decides true if and only if $R(E') <4h^{2k}$.
\end{proof}

%% file: glossary.tex
\nomenclature{$B$}{Budget for the public transport network}
\nomenclature{$\tau_i$}{Travel time per unit distance and passenger for mode $i = 1, \dots, m$}
\nomenclature{$\tau_0$}{Travel time per unit distance and passenger for private transport}
\nomenclature{$\eta_i$}{Energy consumption per unit distance and vehicle for mode $i = 1, \dots, m$}
\nomenclature{$\eta_0$}{Energy consumption per unit distance and vehicle for private transport}
\nomenclature{$c_i$}{Costs per unit distance and vehicle for mode $i = 1, \dots, m$}
\nomenclature{$c_0$}{Costs per unit distance and vehicle for private transport, assumed to be $0$}
\nomenclature{$k_i$}{Capacity per vehicle of mode $i = 1, \dots, m$}
\nomenclature{$k_0$}{Capacity per vehicle of public transport, assumed to be $1$}
\nomenclature{$m$}{The number of modes}

\nomenclature{$D$}{Demand between commodities}
\nomenclature{$L$}{Layout of designed network}
\nomenclature{$w$}{edge weight}
\nomenclature{$d$}{edge routing costs/length}
\nomenclature{$\dist$}{distance between nodes \wrt $d$}
\nomenclature{$F_{\kappa}$}{$s$-$t$-flow for commodity ${\kappa} \in K$}
\nomenclature{$F_e^0$}{number of mode-$0$ passengers on $e$}
\nomenclature{$F$}{Cumulated flow value on edges over all commodities}
\nomenclature{$M_{\kappa}$}{Modal split for commodity ${\kappa} \in K$ describing percentage of travelers using mode $i$ on edge $e$}
\nomenclature{$M$}{Cumulated modal split on edges over all commodities}
\nomenclature{$R$}{Routing costs of (sub)graph}

\nomenclature{$\Tau$}{Total travel time}
\nomenclature{$\Eta$}{Total energy consumption}
\nomenclature{$\Eta^*$}{Minimal value of $\Eta$}
\nomenclature{$\Eta_0$}{Total energy consumption of all-mode-0 shortest path flow solution}